\documentclass[12pt, oneside]{amsart}
\usepackage[margin=1in]{geometry}
\usepackage{graphicx}
\usepackage{amssymb, tikz, latexsym, epsfig, enumerate, multicol, multirow, wasysym, amsmath, amsthm, amscd, amssymb, soul, rotating, pdflscape, float, placeins, dsfont}
\usepackage{biblatex}
\usetikzlibrary{decorations.pathreplacing}
\usepackage[bottom]{footmisc}
\usepackage[colorlinks,citecolor=red,hypertexnames=false]{hyperref}
\hypersetup{%
  colorlinks = true,
  citecolor=red,
  linkcolor  = red
}
\restylefloat{table}
\numberwithin{equation}{section}
\allowdisplaybreaks
\theoremstyle{plain}
\newtheorem{theorem}{Theorem}[section]

\newtheorem{corollary}[theorem]{Corollary}

\theoremstyle{definition}
\newtheorem{definition}[theorem]{Definition}

\newtheorem{example}[theorem]{Example}

\theoremstyle{remark}
\newtheorem{remark}[theorem]{Remark}

\newtheorem{case[theorem]}{Case}

\usepackage{comment}

\title{Limits of Discrete Energy of Families of Increasing Sets} 
\author{Hari Nathan}
\date{January 2025}

\begin{document}

\begin{abstract}
The Hausdorff dimension of a set can be detected using the Riesz energy. Here, we consider situations where a sequence of points, $\{x_n\}$, ``fills in'' a set $E \subset \mathbb{R}^d$ in an appropriate sense and investigate the degree to which the discrete analog to the Riesz energy of these sets can be used to bound the Hausdorff dimension of $E$. We also discuss applications to data science and Erd\H{o}s/Falconer type problems. 
\end{abstract}

\maketitle

\tableofcontents

\section{Introduction and Notation}
\label{section:introduction_and_notation}

\subsection{Introduction}
\label{subsection:introduction}

The Hausdorff dimension of some $E \subset \mathbb{R}^d$ can be defined as (see e.g. \cite{Mattila2015})

$$dim_{\mathcal{H}}(E) = \sup \{s : \text{there is a measure, } \mu \in \mathcal{M}(E)\text{, such that }I_s(\mu) < \infty\}$$

where $\mathcal{M}(E)$ is the set of all Borel probability measures on $E$ and, 

$$I_s(\mu) = \int_E \int_E |x - y|^{-s}d\mu(x)d\mu(y),$$

is known as the Riesz energy of the measure $\mu$. In \cite{Iosevich2014}, a theory of Hausdorff dimension for sequences of finite sets is introduced.\footnote{The article also introduces a theory of Minkowski dimension, but we will not discuss this here.} In particular, given a sequence of sets, $\{P_n\}$, where each $P_n \subset \mathbb{R}^d$ and $\# P_n = n$ is said to be $s$-adaptable if

$$J_s(P_n) = \frac{1}{n(n - 1)}  \sum_{\substack{x, y \in P_n \\x \neq y}} | x - y |^{-s}$$

is uniformly bounded (where $| \cdot |$ is the Euclidean distance). In such a case, the dimension of the sequence is defined as

$$dim_{\mathcal{H}}\{P_n\} = \sup \{s : \{P_n\} \text{ is } s\text{-adaptable}\}.$$

This theory allows one to adapt certain results about continuous problems to the discrete context. For example, in \cite{Iosevich2014} it is shown that, in certain circumstances, if the Falconer distance conjecture is true, then the Erd\H{o}s distance conjecture is also true. In \cite{Senger2011} these results are extended to pairs of distances among triples of points, dot products, etc. 

However, as of yet, no direct connection between $J_s$ and $I_s$ has been discussed. In this paper, we investigate this connection in the following way. Let $E \subset \mathbb{R}^d$ as above, $\{x_n\} \subset E$, and $P_n = \{x_1, ..., x_n\}$. Our goal is to understand, under various circumstances, how

$$\lim_{n \to \infty} J_s(P_n)$$

relates to $dim_{\mathcal{H}}(E)$.

The rest of the article is structured as follows. In the rest of this section, we introduce the notation used throughout the article and discuss connections to data science and other notions of fractal dimension. In section \ref{section:non_randomly_drawn_points} we show how we can use sequences of points to find a lower bound of the Hausdorff dimension of a set. In section \ref{section:randomly_drawn_points} we look at randomly drawn points. Finally, in section \ref{section:connections_to_other_problems} we discuss applications of this theory to Erd\H{o}s/Falconer type problems.

\subsection{Notation}
\label{subsection:notation}

We will use the following standard notation.

\begin{definition}
\label{definition:space_of_probability_measures}
Given a set $E \subset \mathbb{R}^d$, we let $\mathcal{M}(E)$ be the set of all \textbf{Borel probability measures with support in $E$}.
\end{definition}

\begin{definition}
\label{definition:indicator function}
For any set $A$ we define the \textbf{indicator function} as
$$\mathds{1}_A(x) = \begin{cases}
    1 & x \in A \\
    0 & x \notin A.
\end{cases}$$
\end{definition}

\begin{definition}
\label{definition:riesz_kernel_and_enegy}
For any (Borel) measure, $\mu$ on $\mathbb{R}^d$ and $s \geq 0$ we denote \textbf{Riesz potential}, $V_{s}:\mathbb{R}^d \to \mathbb{R}$, as

\begin{equation}
\label{equation:riesz_kernel}
V_s(x; \mu) = V_s(x) = \int_{\mathbb{R}^d} |x - y|^{-s} d\mu(y)
\end{equation}

and the \textbf{Riesz $s$-energy} of $\mu$ as

\begin{equation}
\label{equation:riesz_energy}
I_s(\mu) = \int_{\mathbb{R}^d} V_s(x)d \mu(x) = \int_{\mathbb{R}^d} \int_{\mathbb{R}^d} |x - y|^{-s} d\mu(y) d\mu(x).
\end{equation}
\end{definition}

\begin{definition}
\label{definition:hausdorff_dimension}
Given a measure $\mu$ on $\mathbb{R}^d$ we denote its \textbf{Sobolev dimension} as (see e.g. \cite{Fraser2025-dy})

$$dim_{\mathcal{S}}(\mu) = sup\{s \geq 0 : I_s(\mu) < \infty \}.$$

Given a set, $E \subset \mathbb{R}^d$, we denote it's \textbf{Hausdorff dimension} as $dim_{\mathcal{H}}(E)$ as\footnote{The definition given here is equivalent to the ``standard'' definition in terms of the Hausdorff measure see e.g. \cite{Mattila2015}.}

\begin{equation}
\label{equation:hausdorff_dimension}
dim_{\mathcal{H}}(E) = \sup \{\min\{dim_{\mathcal{S}}(\mu), d\} : \mu \in \mathcal{M}(E) \}.
\end{equation}
\end{definition}

\begin{remark}
\label{remark: sovolev_dimensio_definitions}
It is known (see e.g. \cite{Mattila2015}) that

\begin{equation}
\label{equation:riesz_energy_fourier_form}
I_s(\mu) = \gamma(d, s) \int_{\mathbb{R}^d} | \widehat{\mu}(\xi)|^2 \cdot |\xi|^{s - d} dx.
\end{equation}

where $\widehat{\mu}$ is the Fourier transform of $\mu$ (see e.g. \cite{Mattila2015}) and $\gamma(d, s)$ is a constant that only depends on $d$ and $s$.

The definition of Sobolev dimension normally uses the Riesz energy in this form which allows for dimensions higher than that of the ambient space (see e.g. \cite{Fraser2025-dy}). However, the two formulations for Riesz energy agree for all $0 \leq s \leq d$ and we use the formulation given here since, as we will show, it can be directly estimated using sampling.
\end{remark}

\begin{remark}
\label{remark: other_fractal_dimensions_of_measures}
It is worth briefly distinguishing the definition of $dim_{\mathcal{S}}(\mu)$ from other fractal dimensions of measures. Section 10.1 of \cite{Falconer1997-ot} defines two other notions related to Hausdorff dimension. 

The (lower) Hausdorff dimension is defined as

$$dim_\mathcal{H}(\mu) = \inf\{dim_{\mathcal{H}}(E) : E \text{ is Borel and } \mu(E)  > 0\}.$$

\cite{Sahlsten2020FourierTA} discusses measures with Sobolev dimension less than the (lower) Hausdorff dimension.

The  upper Hausdorff dimension is defined as

$$dim^*_{\mathcal{H}}(\mu) = \inf\{ dim_{\mathcal{H}}(E) : E \text{ is Borel and } \mu(E^c) = 0\}.$$

It is easy to see that the Sobolev and upper Hausdorff dimensions do not match - one need merely take a linear combination of Lebesgue measures of different dimension.
\end{remark}

In addition, we introduce the following notation and definitions which are slight variations on those in \cite{Iosevich2014}.

\begin{definition}
\label{definition:discrete_riesz_energy}
Given a set of $n$ points, $P \subset \mathbb{R}^d$ and $s > 0$, we define the \textbf{discrete $s$-Energy} of $P$ as

\begin{equation}
\label{equation:discrete_riesz_energy}
J_s(P) = \frac{1}{n(n - 1)} \sum_{\substack{x, y \in P \\x \neq y}} |x - y|^{-s}
\end{equation}

where $| \cdot |$ is taken to be the Euclidean distance.
\end{definition} 

\begin{definition}
\label{definition:s_adaptable}
A uniformly bounded sequence of sets, $\{P_n\}$, such that $\# P_n = n$ is \textbf{(Hausdorff) $s$-adaptable} if for all $t < s$

$$J_t(P_n) \lesssim_t 1$$

where $X(n) \lesssim_a Y(n)$ means that there is a finite constant, $C_a$, that depends only on $a$ such that, for all $n$, $X(n) \leq C_a Y(n)$.
\end{definition}

\begin{remark}
\label{remark: equivalent_definitions}
The definition of being $s$-adaptable in Definition \ref{definition:s_adaptable} is essentially equivalent to Definition 2.8 in \cite{Iosevich2014}. In \cite{Iosevich2014} the sets $\{P_n\}$ are required to be $1$-separated instead of uniformly bounded and must satisfy the property that $diam(P_n) \lesssim n^{-1/s}$ (where $diam(E)$ is the diameter of $E$). In such cases, \cite{Iosevich2014} denotes $\{P_n\}$ $s$-adaptable if

$$n^{-2} \sum_{\substack{a, b \in P_n \\a \neq b}} \left(\frac{|a - b|}{diam(P_n)}\right)^{-t} \lesssim_t 1.$$

In our definition, we require the $P_n$ to be uniformly bounded and so $diam(P_n)$ is also uniformly bounded.   
\end{remark}

\begin{example}
\label{example: grid}
Let $\{P_n\}$ be subsets of $[0, 1]$ with

$$P_n = \left\{ \frac{m}{n + 1} : m \in [n] \right\}$$

where $[n] = \{1, ..., n\}$. For $t \in [0, 1)$


\begin{align*}
    J_t(P_n)
    & = \frac{1}{n(n-1)} \sum_{m_1 \in [n]} \sum_{\substack{m_2 \in [n] \\ m_2 \neq m_1}} \left| \frac{m_1 - m_2}{n + 1} \right|^{-t} \\
    & = \frac{2}{n(n-1)} \sum_{m_1 = 1}^{n} \sum_{r = 1}^{m_1 - 1} \left| \frac{r}{n + 1} \right|^{-t} \\
    & \leq \frac{2}{n(n-1)} \sum_{m_1 = 1}^{n} \int_{0}^{m_1} \frac{r^{-t}}{(n+1)^{-t}}dr \\
    & = \frac{2(n+1)^t}{n(n-1)(1-t)} \sum_{m_1 = 1}^{n} m_1^{1 - t} \\
    & \leq \frac{2(n+1)^t}{n(n-1)(1-t)} \int_{0}^{n+1} m_1^{1 -t} dm_1 \\
    & = \frac{2(n+1)^{2}}{n(n-1)(1-t)(2-t)}
\end{align*}

As such, $\{P_n\}$ is $1$-adaptable. In addition, as we shall see below,
\end{example}

\begin{example}
\label{example: cantor set}
(\cite{Betti2025-ny}, Theorem 10)
Start with the interval $[0, 1]$ and positive integers $m, n \in \mathbb{N}$
such that $0 <  m < n$. Divide $[0, 1]$ into $n$ sub-intervals of equal length, and choose $m$ of those intervals. Let $C^{1}_{m, n}$ be the set of endpoints of the intervals chosen. At the $k$-th step, split each of the remaining intervals into $n$ equal sub-intervals, choose m of those intervals, and let $C^{k}_{m, n}$ be the set of endpoints of the intervals chosen. At each step, we pick
the same $m$ sub-intervals from the remaining intervals. Since the $m$ sub-intervals we choose to keep are arbitrary, the set $C^{k}_{m,n}$ is not unique and merely denotes one set satisfying these conditions. In addition, let $C_{m, n} = \lim_{k \to \infty} C^{k}_{m, n}$ (which exists since $C^{k + 1}_{m, n} \supset C^{k}_{m, n}$).

Now, for some $d \in \mathbb{N}$ we take natural numbers $\{m_i\}_{i = 1}^{d}$ and $\{n_i\}_{i = 1}^{d}$ with $0 < m_i < n_i$ we take the set

$$A^k = \prod_{i = 1}^{d}C^{k}_{m_i, n_i} \subset [0, 1]^d$$

and

$$A = \prod_{i = 1}^{d} C_{m_i,n_i}.$$

It is shown in \cite{Betti2025-ny} that $\{A^k\}$ is $s$-adaptable for

$$s = \sum_{i = 1}^{d} \frac{\ln m_i}{\ln n_i} = dim_{\mathcal{H}} \left(A\right).$$
\end{example}

\begin{definition}
Given a sequence of points, $\{x_n\}$ we say it \textbf{generates} the sequence of sets $\{P_n\}$ if $P_n = \{x_1, ..., x_n\}$.
\end{definition}

Notation and definitions only needed for a specific section will be introduced therein.

\subsection{Relationship to Data Science}

Before proceeding with out technical results, we discuss the relationship of this work with data science. Indeed, this relationship was a primary motivation for this work. The ``curse of dimensionality'', a term coined by Richard Bellman in \cite{Bellman2016-hu}, refers to phenomena that occur with data sets in high dimensional spaces. For example, as pointed out in \cite{Alonso2015-zs} in the context of nearest-neighbor algorithms

\begin{quote}
as the dimensionality of the data increases, the longest and shortest distance between points tend to become so close that the distinction between ”near” and ”far” becomes meaningless.
\end{quote}

Another manifestation (see e.g. \cite{Hastie2009-km}) can be seen by considering a data set as $N$ points in $[0, 1]^d$. In this framework, in order for a machine learning algorithm to perform well, one might want $k$ points in each hyper-cube with side length $\delta$ for $\delta \ll 1$. Thus, one needs $\sim k \delta^d$ data points.

These problems arise naturally in fields such as natural language processing where documents are often represented as word vectors and, as such, are extremely high-dimensional (see e.g. \cite{Karlgren2008-ye}). 

It is thus of great interest in data science to be able to estimate the ``true'' dimension of a data set. \cite{Gneiting2012-yi} surveys various methods for doing so but provides no theoretical guarantees for any of these methods. \cite{Alonso2015-zs} approaches this issue by defining a non-trivial notion of Hausdorff dimension for finite sets and show that, under certain conditions, one gets convergence to the Hausdorff dimension.

\cite{Betti2025-ny} proposed using the definition of Hausdorff dimension similar the one we discuss in this paper. They propose defining the Hausdorff dimension of a family of points sets, $\{P_n\}$ in $\mathbb{R}^d$ as\footnote{Actually,\cite{Iosevich2014} and \cite{Betti2025-ny} use a slight variation of $J_s$ where they divide by $n^{-2}$ instead of $(n(n-1))^{-1}$ though this is clearly asymptotically irrelevant.}

$$dim_\mathcal{H}\left( \{P_n\} \right) = \sup \left\{ s \in [0, d] : \sup_{n} J_s(P_n) < \infty \right\}.$$

Of particular relevance to this paper, they proved that if there is a compact set, $E$, such that the $P_n \subset E$ then $dim_\mathcal{H}\left( \{P_n\} \right)$ is bounded above by the upper Minkowski dimension of $E$. In addition, they proved various results relating $J_s(P_n)$ to the possibility of \textit{dimensionality reduction}.

Dimensionality reduction refers to a wide family of techniques that represent high dimensional data as lower dimensional data while (hopefully) retaining important features of the data. For example, one could project points $x_1, ..., x_n \in \mathbb{R}^d$ onto some $d'$-plane for $d' < d$. The hope is that by finding the correct plane (for the given data set) the projected points would retain sufficient explanatory power for the task at hand.The data science literature has many other, and generally more effective, dimensionality reduction techniques (see e.g. \cite{Shalev-Shwartz2014-vb}).

\subsection{Connections to Other Fractal Dimensions}

In this subsection, we briefly look connections to the general theory of fractal dimension.

For $E \subset \mathbb{R}^d$ let $N(\ell)$ be the minimum the number of boxes of side length $\ell$ needed the cover $E$. The Minkowski dimension (also known as the Minkowski–Bouligand dimension or box-counting dimension) of $E$ is defined as

$$dim_{\mathcal{M}}(E) = \lim_{\ell \to 0} \frac{\log N(\ell)}{\log(1/\ell)}$$

where the limit exists. The Minkowski dimension has been used in geology (e.g. \cite{Husain2021-ya}), medicine (e.g. \cite{Ferreira2006-th}), and  astronomy (e.g. \cite{Caicedo-Ortiz2015-gk}) (among other applications). The Minkowski dimension is convenient for applications because one gets the same result if $N(\ell)$ is replaced with estimates of $N(\ell)$ (e.g. the number standard dyadic boxes that intersect $E$) which are often quite easy to compute. The theorems proved in this paper can, hopefully, provide similarly practical methods to compute the Hausdorff dimension of data sets.

In a different direction, \cite{Fraser2024-ab} has recently introduced the Fourier spectrum which generalizes the s-energy in order to connect the notions of Hausdorff and Fourier dimensions. It generalizes formula (\ref{equation:riesz_energy_fourier_form}) energy with a parameter $\theta \in [0, 1]$ to

\begin{equation}
\label{equation:fourier_spectrum_fourier_form}
\mathcal{J}_{s, \theta}(\mu) =  \left( \int_{\mathbb{R}^d} | \widehat{\mu}(\xi) |^{2/\theta} \cdot |\xi|^{s/\theta - d} dx\right)^{\theta}
\end{equation}

where, by convention,

$$\mathcal{J}_{s, 0}(\mu) = \sup_{\xi \in \mathbb{R}^d} | \widehat{\mu}(\xi)|^2 \cdot |\xi|^{s}.$$

Similarly to definition \ref{definition:hausdorff_dimension}, \cite{Fraser2024-ab} defines the Fourier Spectrum of $\mu$ at $\theta$ as

$$dim_F^{\theta}(\mu) = sup\{s \geq 0 : \mathcal{J}_{s, \theta}(\mu) < \infty \}$$

and, for $E \subset \mathbb{R}^d$

$$dim_{E}^{\theta} = \sup\{\min\{ dim_F^{\theta}(\mu), d \} : supp(\mu) \subset E \}.$$

Our work in this paper directly uses the form of the equation for Riesz energy in equation (\ref{equation:riesz_energy}) as opposed to the form in equation in (\ref{equation:riesz_energy_fourier_form}). As such, it seems difficult to apply directly to the Fourier spectrum as defined in equation (\ref{equation:fourier_spectrum_fourier_form}). However, if we can find a form of equation (\ref{equation:fourier_spectrum_fourier_form}) similar to that of (\ref{equation:riesz_energy}), then the methods here might apply to this as well.

\section{Points Which Converge to a Measure}
\label{section:non_randomly_drawn_points}

\FloatBarrier

In this section we discuss situations where a sequence of points, $\{x_n\}$, in some sense, converge to a measure $\mu \in \mathcal{M}(E)$ for some $E \subset \mathbb{R}^d$. For example, if $E = [0, 1]^d$ and the $\{x_n\}$ represented points on increasingly fine grids of $[0, 1]^d$ such that for $k \in \mathbb{N}$

$$P_{2^{kd}} = \left\{\left(\frac{(i_1,..., i_d)}{2^k}\right) : 0 \leq i_1, ..., i_d < 2^k\right\}$$

(see Figure \ref{figure:grid}). In this case, if we let $\mu$ be the Lebesgue measure on $[0, 1]^d$, we might expect that $J_s(P_{2^{kd}}) \to I_s(\mu)$ as $k \to \infty$.

\begin{figure}
\begin{tikzpicture}
\foreach \x in {0,1,...,3} {
    \foreach \y in {0,1,...,3} {
        \fill[color=black] (0.5 * \x, 0.5 * \y) circle (0.04);
    }
}

\draw[thick,->] (2.5,1) -- (3,1);

\foreach \x in {0,1,...,7} {
    \foreach \y in {0,1,...,7} {
        \fill[color=black] (4 + 0.25 * \x, 0.25 * \y) circle (0.04);
    }
}

\draw[thick,->] (6.5,1) -- (7,1);

\foreach \x in {0,1,...,15} {
    \foreach \y in {0,1,...,15} {
        \fill[color=black] (8 + 0.125 * \x, 0.125 * \y) circle (0.04);
    }
}

\draw[thick,-] (10.25,1) -- (10.50,1);
\foreach \x in {1, 2, 3} {
    \fill[color=black] (10.50 + \x * .75/4, 1) circle (0.02);
}
\draw[thick,->] (11.25,1) -- (11.50,1);

\filldraw[fill=black, draw=black] (12,0) rectangle (14,2);

\end{tikzpicture}
\caption{A sequence of points filling in the unit square via increasingly fine grids.}
\label{figure:grid}
\end{figure}
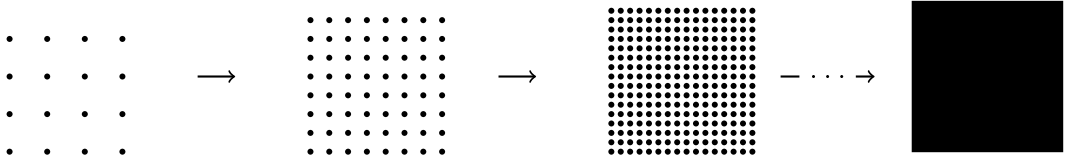

In the next theorem, we show that, in the most general case, very little can be said about the energies of a sequence of points.

\begin{theorem}
\label{theorem:any_sequence_of_energies}
Let $s \in (0, d]$. Let $\{e_k\}$ be a sequence of positive real numbers. Then there is a sequence of points, $\{x_n\} \subset \mathbb{R}^d$ which generates $\{P_n\}$ such that there is an increasing subsequence $\{n_k\}$ with $J_s(P_{n_k}) = e_k$.
\end{theorem}

\begin{proof}
We proceed via an inductive construction. Set $x_1$ as the origin and then $x_2$ such that $|x_1 - x_2|^{-s} = e_1$. This gives us $n_1 = 2$.

Before proceeding with the construction, note that if we add a point $x_{n_k + 1}$ to $P_{n_k}$ we can put $x_{n_k + 1}$ far away enough form $P_{n_k}$ so that for any $\varepsilon > 0$,

$$\frac{(n_k)(n_k - 1)}{(n_k + 1)(n_k)} J_s(P_{n_k}) = \frac{n_k - 1}{n_k + 1}J_s(P_{n_k}) < J_s(P_{n_k + 1}) < (1 + \varepsilon)\frac{n_k - 1}{n_k + 1}J_s(P_{n_k}).$$

Now, if $e_{k + 1} \geq e_{k}$ then we can pick a point, $y \in\mathbb{R}^d$ such that $J_s(P_{n_k} \cup \{y\}) < J_s(P_{n_k})$ and such that the line segment containing $y$ to $x_1$ does not pass through any other points in $P_{n_k}$. As we move $y$ closer to $x_1$, $J_s(P_{n_k} \cup \{y\}) \to \infty$. Thus, somewhere along that line segment, $J_s(P_{n_k} \cup \{y\}) = e_{k + 1}$ and we have $n_{k + 1} = n_{k} + 1$.

On the other hand, if

$$\frac{n_k - 1}{n_k + 1} e_k <  e_{k + 1} < e_{k}$$

one can place $x_{n_k + 1}$ far enough away from $P_{n_k}$ so that $J_s(P_{n_k + 1}) < e_{k + 1}$ and then follow a similar procedure to above to make $J_s(P_{n_k + 1}) = e_{k + 1}$. Finally, if

$$e_{k + 1} < \frac{n_k - 1}{n_k + 1} e_k$$

for some very small $\varepsilon > 0$, one can place $x_{n_k + 1}$ far enough away such that

$$J_s(P_{n_k + 1}) = (1 + \varepsilon)\frac{n_k - 1}{n_k + 1} e_{k}.$$

Repeating this procedure, one gets

$$J_s(P_{n_k + \ell}) = (1 + \varepsilon)^{\ell} \frac{(n_k - 1)(n_k)}{(n_k + \ell - 1)(n_k + \ell)} e_k.$$

There is a minimum $\ell > 0$ such that

$$\frac{(n_k - 1)(n_k)}{(n_k + \ell - 1)(n_k + \ell)} e_k < e_{k + 1}.$$

By picking $\varepsilon$ appropriately, one gets

$$J_s(P_{n_k + \ell}) = e_{k +1}$$

making $n_{k + 1} = n_k + \ell$.
\end{proof}

In the proof of theorem \ref{theorem:any_sequence_of_energies}, one might notice that the points can become arbitrarily far away and, in general, have little relation to each other or any other ambient set. As such, theorem \ref{theorem:any_sequence_of_energies} it is not necessarily surprising.  In order to have the $\{J_s(P_n)\}$ mean something we need to add some structure.

\subsection{Normalized counting measures with weak star convergence}
\label{subsection:normalized_counting_measure_convergence}

The first type of structure we add is perhaps the most natural i.e. we use the normalized counting measure and require that these measures converge.

\begin{definition}
\label{definition:weak_star_convergence} (see e.g. \cite{Borodachov2019})
A sequence of measures, $\{\mu_n\}$ on a set $A$ is said to \textbf{weak-$\star$ converge} to a measure $\mu$ on $A$, written

$$\mu_n \overset{\star}{\longrightarrow} \mu$$

if for any bounded continuous function, $f:A \to \mathbb{R}$,

$$\int_A f(x)d\mu_n(x) \to \int_{A}f(x)d\mu(x).$$
\end{definition}

It is known, (see e.g. Lemma 1.6.6 in \cite{Borodachov2019}), that if $\mu_n \overset{\star}{\longrightarrow} \mu$ then we get the same for the product measures i.e. $\mu_n \times \mu_n \overset{\star}{\longrightarrow} \mu \times \mu.$

\begin{definition}
\label{definition:normalized_counting_measure}
Given a sequence of points $\{x_n\}$ that generate $\{P_n\}$, the \textbf{induced sequence of normalized counting measures} is $\{\nu_n\}$ where

\begin{equation}
\label{equation:induced_normalized_counting_measures}
\nu_n = \frac{1}{n}\sum_{a \in P_n} \delta_a
\end{equation}

where $\delta_a$ is the dirac delta measure at $a$.
\end{definition}

\begin{theorem}
\label{theorem:normalized_counting_weak_star_liminf}
Suppose that $\nu_n$ converges weakly to a Borel probability measure $\mu$ on $\mathbb R^d$. Then

$$I_s(\mu) \leq \liminf_{n\to\infty} J_s(P_n),$$

where both sides are allowed to take the value $+\infty$. In particular, if $I_s(\mu)=+\infty$, then
$$J_s(P_n) \longrightarrow +\infty.$$
\end{theorem}

\begin{proof}
For $M>0$, define

$$K_M(x,y)=\min\{M,|x-y|^{-s}\}.$$

The function $K_M$ is bounded and continuous on $\mathbb R^d\times\mathbb R^d$, where its value on the diagonal is $M$. Since weak convergence is preserved under products,

$$\iint K_M(x,y)\,d\nu_n(x)d\nu_n(y) \longrightarrow \iint K_M(x,y)\,d\mu(x)d\mu(y).$$

On the other hand,

\begin{align*}
    \iint K_M(x,y) d\nu_n(x) d\nu_n(y)
    & = \frac{1}{n^2}\sum_{i = 1}^{n} K_M(x_i, x_i) + \frac{1}{n^2} \sum_{\substack{1\leq i,j\leq n\\i\neq j}} K_M(x_i,x_j) \\
    & \leq \frac{M}{n} + \frac{n-1}{n}J_s(P_n).
\end{align*}
It follows that

$$\iint K_M(x,y)\,d\mu(x)d\mu(y) \leq \liminf_{n\to\infty}J_s(P_n).$$

Letting $M\to\infty$ and using monotone convergence gives

$$I_s(\mu) \leq \liminf_{n\to\infty}J_s(P_n).$$

If $I_s(\mu) = +\infty$, the right-hand side is infinite. Equivalently, for every $A > 0$, all sufficiently large $n$ satisfy $J_s(P_n) > A$.
\end{proof}

\begin{theorem}
\label{theorem:normalized_counting_weak_star_unif_int}
Suppose that $\nu_n$ converges weakly to a Borel probability measure $\mu$ and satisfies the uniform integrability condition

$$\lim_{M\to\infty} \sup_{n\geq2} \frac1{n(n-1)} \sum_{\substack{1\leq i,j\leq n\\i\neq j}}
\left(|x_i-x_j|^{-s}-M\right)_+ = 0.$$

Then

$$J_s(P_n)\longrightarrow I_s(\mu).$$
    
\end{theorem}

\begin{proof}
For $M>0$, set

$$J_{s,M}(P_n) = \frac{1}{n(n-1)} \sum_{\substack{1\leq i,j\leq n\\i\neq j}} K_M(x_i ,x_j).$$

The identity

$$\iint K_M\,d\nu_n d\nu_n = \frac{M}{n}+\frac{n-1}{n}J_{s,M}(P_n)$$

and weak convergence imply that, for each fixed $M$,

\begin{equation}
    \label{equation:JsM_to_integral_KM}
    J_{s,M}(P_n) \longrightarrow \iint K_M(x,y)\,d\mu(x)d\mu(y).
\end{equation}

Moreover,

$$0\leq J_s(P_n)-J_{s,M}(P_n) = \frac1{n(n-1)} \sum_{i\neq j} \left(|x_i-x_j|^{-s}-M\right)_+.$$

By the uniform integrability assumption, for any $\varepsilon > 0$, for big enough $M$, we get that for all $n$,

$$J_s(P_n) - J_{s, M}(P_n) \leq \varepsilon.$$

As such, for any $\varepsilon > 0$, for sufficiently large $M$ we have that, for all $n$

$$J_s(P_n) \leq J_{s, M}(P_n) + \varepsilon \leq M + \varepsilon \leq \infty.$$

So, by Theorem \ref{theorem:normalized_counting_weak_star_liminf}, $I_s(\mu) < \infty$. For this $M$ and sufficiently large  $n$ we get

$$J_{s}(P_n) \geq J_{s,M}(P_n)\geq \iint K_M(x,y)d\mu(x)d\mu(y) - \varepsilon \geq I_s(\mu) - 2\varepsilon$$

and

$$J_{s}(P_n) \leq J_{s,M}(P_n) + \varepsilon \leq \iint K_M(x,y)d\mu(x)d\mu(y) + 2\varepsilon \leq I_s(\mu) + 3\varepsilon$$

where we used equation (\ref{equation:JsM_to_integral_KM}) and where the last inequality in each comes from the monotone convergence theorem and the finiteness of $I_s(\mu)$. This gets us

$$I_s(\mu) - 2\varepsilon \leq J_s(P_n) \leq I_s(\mu) + 3\varepsilon.$$

Letting $\varepsilon \to 0$ proves the theorem.
\end{proof}

\begin{corollary}
\label{corollary:t_adpatable_for_t_gt_s}
Suppose that $\nu_n$ converges weakly to $\mu$. If there is an exponent $t>s$ such that

$$\sup_{n\geq2}J_t(P_n)<\infty,$$

then

$$J_s(P_n) \overset{n \to \infty}{\longrightarrow}I_s(\mu).$$
\end{corollary}

\begin{proof}
For any $M$, letting $r^{-s} = M$,
\begin{align*}
    \frac{1}{n(n-1)} & \sum_{\substack{1\leq i,j\leq n\\i\neq j}} \left(|x_i-x_j|^{-s}-M\right)_+ \\
    & \leq \frac{1}{n(n-1)} \sum_{0 < |x_i - x_j| < r} |x_i - x_j|^{-s} \\
    & \leq \frac{1}{n(n-1)} r^{t - s}\sum_{0 < |x_i - x_j| < r} |x_i - x_j|^{-t} \\
    & \leq r^{t - s} J_t(P_n).
\end{align*}

This goes to zero as $r \to 0$ (i.e. as $M \to \infty$) which gives us the uniform integrability condition in Theorem \ref{theorem:normalized_counting_weak_star_unif_int}.
\end{proof}

\begin{corollary}
\label{corollary:normalized_counting_s_adaptable_dimension}
Let $\{x_n\} \subset \mathbb{R}^d$ generate $\{P_n\}$ and induce measures $\{\nu_n\}$ as in definition \ref{definition:normalized_counting_measure}. If $\{P_n\}$ is $s$-adaptable and $\nu_n \overset{\star}{\longrightarrow} \mu$, for some $\mu$, then

$$dim_\mathcal{H}(supp(\mu)) \geq s.$$
\end{corollary}

\begin{proof}
For any $s' < s$, we take $t \in (s', s)$ and we know that

$$\sup_{n \geq 2} J_t(P_n) < \infty.$$

So, by corollary \ref{corollary:t_adpatable_for_t_gt_s}, $I_{s'}(\mu) < \infty$. Thus, by Theorem 2.8 in \cite{Mattila2015}, we know that $dim_\mathcal{H}(supp(\mu)) \geq s'$.  
\end{proof}

It is possible to find a sequence of points such that the energies converge but the induced normalized counting measure don't converge to any given measure, though constructing a counter-example is non-trivial. We give one here.

\begin{example}
\label{example: energy_converge_measure_doesnt} The key behind this example is to place points such that they approximate a sequence of measures that gradually change but never converge. Here, we consider the unit circle, $C \subset \mathbb{R}^2$, parametrized by the angle in the standard way, and measures $\{\gamma_n\} \subset \mathcal{M}(C)$ such that

$$\gamma_n = \alpha + \beta_{H_n}$$

where $\alpha$ is half of the uniform probability measure on the $C$, $H_n = \sum_{k = 1}^{n} \frac{1}{k}$, and $\beta_t$ is the uniform probability measure on the semi-circle $(t - \pi/2, t+\pi/2)$. We can see that, since the $\gamma$'s are just rotations of each other and since $I_s$ is invariant under rigid transformations, we get that

$$\forall m, n \in \mathbb{N}, \forall s \in [0, 1), I_s(\gamma_m) = I_s(\gamma_n) < \infty.$$

Let $\{x_k\}$ fill in $supp(\alpha) = C$ in a uniform manner and $\{y_k\}$ fill in $supp(\beta_0)$ in a uniform manner. We will place the points (some rotated) in the following manner. Let $P_k$ be the set of the first $k$ points placed.

Phase $0$: For some large $N$, place $\{x_1, ..., x_N\}$ and $\{y_1, ..., y_N\}$, alternating between the $x$'s and $y$'s, resulting in a close approximation of $\gamma_0$.

Phase $n + 1$: Assume that in phases $1$ to $n$, a total of $2 k_n$ were used and approximate $\gamma_n$. Note that there are two only two wedges whose measure changes between $\gamma_n$ and $\gamma_{n + 1}$: $A_1$ whose probability decreases and $A_2$ whose probability increases (see Figure \ref{fig:circle_measures}). Starting with $x_{k_n + 1}$ and $y_{k_n + 1}$, place the points in $\{x_k\}$ and $\{y_k\}$, alternating between the two in the same way we did in phase $n$ except that any $x$'s or $y$'s that would have gone to $A_1$ now get rotated to go to the equivalent place in $A_2$. We do this as long as $\nu_k(A_2)$ is getting closer to $\gamma_{n + 1}(A_2)$ and then stop. (This will happen eventually since $A_2$ is getting $x$'s at twice the rate as any other part of the circle and $y$'s at a rate no lower than any other part of the circle). At the end of this phase, the $2 k_{n + 1}$ points are now a better approximation of $\gamma_{n + 1}$ then the $2 k_n$ points were of $\gamma_n$ and, as such,

$$|J_s(P_{2k_{n}}) - I_s(\gamma_{n})|$$

is decreasing as a function of $n$. As we proceed through phase $n + 1$, the discrete $s$-energy at any point is

\begin{align*}
    J_s(P_{k}) = \frac{1}{2k(2k - 1)} \left[ \sum_{\substack{x, y \in \widetilde{B}_1 \cup \widetilde{B}_2 \\x \neq y}} |x - y|^{-s} + 2 \left( \sum_{\substack{x \notin \widetilde{B}_1 \cup \widetilde{B}_2 \text{ or } y \notin \widetilde{B}_1 \cup \widetilde{B}_2 \\x \neq y}} |x - y|^{-s} \right) \right]
\end{align*}

where $\widetilde{B}_i = B_i \cap P_{k_n}$ (for $i = 1, 2)$. Since the proportion of pairs of summands in $\widetilde{B}_1 \cup \widetilde{B}_2$ is $(\frac{k - 2}{k})^2 \to 1$, as $n$ increases, the second term above goes to zero and so we get

$$J_s(P_k) \overset{k \to \infty}{\longrightarrow} I_s(\gamma_0).$$

Finally, we note that since $\gamma_n = \alpha+\beta_{H_n}$ and $H_n \to \infty$, the measures never converge since the midpoint of the semi-circle with bigger measure never converges.

\begin{figure}
\centering

\begin{tikzpicture}
\draw (-3,0) circle (2.0cm);
\draw[line width=3pt] (-3,0) ++(-90:2cm) arc[start angle=-90, end angle=90, radius=2.0cm];
\node at (-2.5,-1.5) {\large $A_1$};
\node at (-3.5,1.5) {\large $A_2$};
\node at (-1.5,0.5) {\large $B_1$};
\node at (-4.5,-0.5) {\large $B_2$};

\draw[->, ultra thick] (-0.75,0) -- (0.75,0);

\draw (3,0) circle (2cm);
\draw[line width=3pt] (3,0) ++(-60:2cm) arc[start angle=-60, end angle=120, radius=2cm];
\node at (3.5,-1.5) {\large $A_1$};
\node at (2.5,1.5) {\large $A_2$};
\node at (4.5,0.5) {\large $B_1$};
\node at (1.5,-0.5) {\large $B_2$};

\end{tikzpicture}

\caption{Two measures, $\gamma_{n - 1}$ and $\gamma_{n}$ on the circle. The thick parts indicate the support of $\beta_{n - 1}$ and $\beta_n$. As we go from $\gamma_{n}$ to $\gamma_{n + 1}$, $A_1$ decreases in measure, $A_2$ correspondingly increases in measure, and the measures of the $B_i$ remain unchanged.}
\label{fig:circle_measures}
\end{figure}
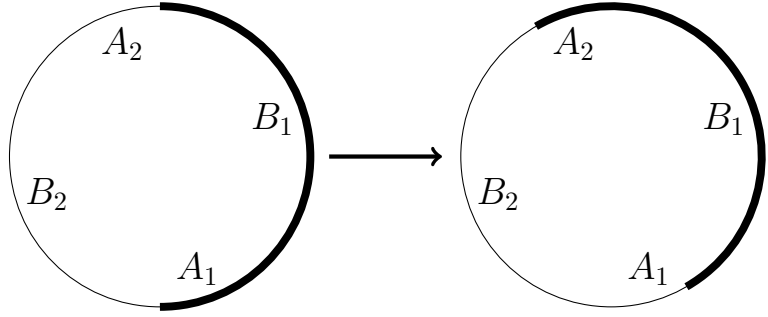

\end{example}

\subsection{Measures induced by balls}
\label{subsection:measure_generated_by_balls}

In sub-section \ref{subsection:normalized_counting_measure_convergence} we looked at the normalized counting measure generated by a finite set of points. Here, we look another way to generate probability measures given a sequence of points - i.e. by putting balls around each point with the radius of the balls decreasing as we add more points.

\begin{definition}
\label{definition:ball_measures}
Given a sequence of points $\{x_n\}$ that generate $\{P_n\}$ and some $s, c > 0$, the \textbf{induced sequence of ball measures} is $\{\kappa_{n, c}\}$ where

\begin{equation}
\label{equation:ball_measures}
d\kappa_{n, c} = n^{-1} \omega_d^{-1} c^{-d} n^{d/s} \sum_{a \in P_n} \mathds{1}_{B(a, cn^{-1/s})}(x) dx
\end{equation}

where $\omega_d$ is the volume of the unit $d$-ball.

\end{definition}

This measure is used in e.g. \cite{Iosevich2014} and \cite{Iosevich2019}.  We here show that $I_s(\kappa_{n,c})$ is also related to $J_s(P_n)$. The following theorem is a tighter version of Theorem 2.7 in \cite{Iosevich2014}.

\begin{theorem}
\label{theorem:J_s_vs_I_s_for_ball_measures} For a set $P_n \subset \mathbb{R}$ with $n$ elements where, for $a \neq b \in P_n$, $|a - b| > 2cn^{-1/s + \varepsilon}$ for some $\varepsilon > 0$, $n > 2^{s + 1}$, and $\kappa_{n, c}$ defined as in equation (\ref{equation:ball_measures}),

$$I_s(\kappa_{n,c}) = \left( 1 + O(n^{-\varepsilon}) \right) \frac{n - 1}{n} J_s(P_n) + \frac{\sigma_d 2^{d - s}}{\omega_dc^s(d - s)}$$

where $\sigma_d$ is the surface area of the unit $d$-ball.
\end{theorem}

\begin{proof}
For $I_s(\kappa_{n, c})$ we have

$$I_s(\kappa_{n,c}) = n^{-2} \omega_d^{-2} c^{-2d} n^{2d/s} \sum_{a, b \in P_n} \int_{B(a, cn^{-1/s)}}\int_{B(b, cn^{-1/s)}} |x - y|^{-s}dydx.$$

In the summation, there are $n$ terms where $a = b$. For these, we can substitute $u = x - y$ and $v = x$ to compute the sum of these integrals as below.

\begin{align*}
    n^{-2} \omega_d^{-2} c^{-2d} n^{2d/s} \sum_{a \in P_n} & \int_{B(a, cn^{-1/s})}\int_{B(a, cn^{-1/s})} |x - y|^{-s} dy dx \\
    & = n^{-1} \omega_d^{-2} c^{-2d} n^{2d/s} \int_{B(0, cn^{-1/s})}\int_{B(0, cn^{-1/s})} |x - y|^{-s} dy dx \\
    & = n^{-1} \omega_d^{-2} c^{-2d} n^{2d/s}\int_{B(0, cn^{-1/s})}\int_{B(0, 2cn^{-1/s})} |u|^{-s} du dv \\
    & = n^{-1} \omega_d^{-1} c^{-d} n^{d/s} \int_{B(0, 2cn^{-1/s})} |u|^{-s} du \\
    & = n^{-1} \omega_d^{-1} c^{-d} n^{d/s} \int_{0}^{2cn^{-1/s}} r^{d - 1 -s} \sigma_d dr \\
    & = \frac{\sigma_d 2^{d - s}}{\omega_d c^s (d - s)}.
\end{align*}

To bound the $n(n - 1)$ terms where $a \neq b$ we take the power series of $f(t) = (T + t)^{-s}$ which is

\begin{equation}
    \label{equation:power_series}
    f(t) = T^{-s} + \sum_{k = 1}^{\infty} (-1)^k\frac{\prod_{j = 0}^{k - 1}(s-j)}{k!}T^{-s-k}t^k = T^{-s} + \sum_{k = 1}^{\infty}D_kt^k.
\end{equation}

This has radius of convergence $T$. Let

$$\zeta = n^{-2}\omega_d^{-2} c^{-2d} n^{2d/s}$$

and take the terms to get

\begin{align*}
    & \zeta \sum_{\substack{a, b \in P_n \\a \neq b}} \int_{B(a,cn^{-1/s})} \int_{B(b,cn^{-1/s})} |x - y|^{-s} dy dx \\
    & = \zeta \sum_{\substack{a, b \in P_n \\a \neq b}} |a - b|^{-s} \int_{B(a,cn^{-1/s})} \int_{B(b,cn^{-1/s})} 1 + \sum_{k = 1}^{\infty} D_k |a-b|^{-k} ||x-y|-|a-b||^k dxdy \\
    & = \zeta \sum_{\substack{a, b \in P_n \\a \neq b}} |a - b|^{-s} \int_{B(a,cn^{-1/s})} \int_{B(b,cn^{-1/s})} 1 + \sum_{k = 0}^{\infty} D_k \frac{||x-y|-|a-b||^k}{|a-b|^{k}} dxdy \\
    & \leq \zeta \sum_{\substack{a, b \in P_n \\a \neq b}} |a - b|^{-s} \int_{B(a,cn^{-1/s})} \int_{B(b,cn^{-1/s})} 1 + \sum_{k = 0}^{\infty} D_k \frac{\left(2n^{-1/s}\right)^k}{\left(2n^{-1/s + \varepsilon}\right)^k} dxdy \\
    & = \zeta \sum_{\substack{a, b \in P_n \\a \neq b}} |a - b|^{-s} \int_{B(a,cn^{-1/s})} \int_{B(b,cn^{-1/s})} 1 + O(n^{-\varepsilon}) \\
    & = \left( 1 + O(n^{-\varepsilon}) \right) \frac{n - 1}{n} J_s(P_n).
\end{align*}
\end{proof}

With theorem \ref{theorem:J_s_vs_I_s_for_ball_measures} we get a corollary similar to corollary \ref{corollary:normalized_counting_s_adaptable_dimension}.

\begin{corollary}
\label{corollary:ball_measures_s_adaptable_dimension}
Let $\{x_n\} \subset \mathbb{R}^d$ generate $\{P_n\}$ and induce measures $\{\kappa_{n, c}\}$ as in definition \ref{definition:ball_measures}. If $\{P_n\}$ is $s$-adaptable and  $\kappa_{n,c} \overset{\star}{\longrightarrow} \mu$, for some $\mu$, then

$$dim_\mathcal{H}(supp(\mu)) \geq s.$$
\end{corollary}

The proof is essentially the same as the proof of corollary \ref{corollary:normalized_counting_s_adaptable_dimension}.

\section{Randomly Drawn Points}
\label{section:randomly_drawn_points}

In this section, we investigate the situation where $\mu \in \mathcal{M}(E)$ for some set $E$ and we generate points $\{X_n\}$ where the $X_n$ are drawn IID via $\mu$. Note that, in this case, the $\{ X_n \}$ are random variables and so if they generate $\{P_n\}$ then the $\{ J_s(P_n) \}$ are also random variables. One might expect that, in a probabilistic sense,

$$J_s(P_n) \overset{n \to \infty}{\longrightarrow} I_s(\mu).$$

This is indeed the case in various senses which we explore in this section.

\subsection{Basic statistics}
\label{subsection:basic_statistics}
We first calculate some basic statistics. Recall that, given a probability space, $(\Omega, \mathcal{F},\mu)$, and a random variable $X:\Omega \to \mathbb{R}$, the expected value of $X$ is

$$\mathbb{E}[X] = \mathbb{E}X = \int_\Omega X(\omega) d\mu(\omega)$$

and the variance of $X$ is

$$Var(X) = Var X = \mathbb{E}\left[(X - \mathbb{E}[X])^2\right] = \mathbb{E} \left[X^2\right] - \mathbb{E}[X]^2.$$

\begin{theorem}
\label{theorem:expected_value_discrete_energy}
Let $\left( \mathbb{R}^d, \mathcal{B}, \mu \right)$ be a Borel probability space on $\mathbb{R}^d$ and $\{X_n\}$ be IID random variables drawn via $\mu$, that generate $\{P_n\}$. Then, for all $n \geq 1$,

$$\mathbb{E} \left[J_s(P_n)\right] = I_s(\mu).$$
\end{theorem}

\begin{proof}
\begin{align*}
    \mathbb{E}\left[ J_s(P_n) \right]
    & = \mathbb{E}\left[ \frac{1}{n(n - 1)}\sum_{\substack{1 \leq j, k \leq n \\ j \neq k}} |X_j - X_k|^{-s} \right] \\
    & = \frac{1}{n(n - 1)}\sum_{\substack{1 \leq j, k \leq n \\ j \neq k}} \mathbb{E}\left[ |X_j - X_k|^{-s} \right] \\
    & = \frac{1}{n(n - 1)} \cdot n(n - 1) \int \int |x - y|^{-s} d\mu(x) d\mu(y) \\
    & = I_s(\mu).
\end{align*}
\end{proof}

\begin{theorem}
\label{theorem:variance_discrete_energy}
Let $\left( \mathbb{R}^d, \mathcal{B}, \mu \right)$ be a Borel probability space on $\mathbb{R}^d$, $\{X_n\}$ be IID random variables drawn via $\mu$, that generate $\{P_n\}$. Then $Var(J_s(P_n)) = \infty$ iff $I_{2s}(\mu) < \infty$.
\end{theorem}

\begin{proof}
$$Var \left[J_s(P_n)\right] = \mathbb{E}\left[ J_s(P_n)^2 \right] - \mathbb{E}\left[ J_s(P_n) \right]^2.$$

From theorem \ref{theorem:expected_value_discrete_energy}, the second term is $I_s(\mu)^2$. For the first term, we get

$$\mathbb{E}\left[ J_s(P_n)^2 \right] = \sum_{\substack{1 \leq i, j  \leq n \\i \neq j}} \sum_{\substack{1 \leq i',  j'  \leq n\\i' \neq j'}} \mathbb{E}\left[|X_i - X_j|^{-s} |X_{i'} - Y_{j'}|^{-s} \right].$$

This has three types of terms. \begin{itemize}
    \item [i)] Terms where $|\{i, j\} \cap \{i', j'\}| = 2$.
    \item [ii)] Terms where $|\{i, j\} \cap \{i', j'\}| = 1$.
    \item [iii)] Terms where $|\{i, j\} \cap \{i', j'\}| = 0$.
\end{itemize}

For type i) we can compute that

$$\mathbb{E}\left[|X_i - X_j|^{-s} |X_i - X_j|^{-s} \right] = \mathbb{E}\left[|X_i - X_j|^{-2s} \right] = I_{2s}(\mu).$$

For type iii), since all of the points are IID, we can compute that

$$\mathbb{E}\left[|X_i - X_j|^{-s} |X_{i'} - X_{j'}|^{-s} \right] = \mathbb{E}\left[|X_i - X_j|^{-s}\right] \cdot \mathbb{E}\left[|X_{i'} - X_{j'}|^{-s}\right] = \mathbb{E}\left[|X_i - X_j|^{-s} \right]^2 = \left(I_s(\mu)\right)^2.$$

For type ii) we can compute, using Cauchy-Schwarz, that

\begin{align*}
    \mathbb{E}\left[|X_i - X_j|^{-s} |X_i - X_{j'}|^{-s} \right]
    & = \int \int \int |x - y|^{-s} |x - y'|^{-s} d\mu(y') d\mu(y) d\mu(x) \\
    & = \int \left( \int |x - y|^{-s} d\mu(y) \right)^2 d\mu(x) \\
    & \leq \int  \left( \int d\mu(y) \right) \left( \int |x - y|^{-2s} d\mu(y) \right) d\mu(x) \\
    & = \int V_{2s}(x; \mu) d\mu(x) \\
    & = I_{2s}(\mu).
\end{align*}

Since $I_{2s}(\mu) < \infty \Rightarrow I_s(\mu) < \infty$, we get that

$$Var(J_s(P_n)) <\infty
\Leftrightarrow \mathbb{E}\left[ J_s(P_n)^2 \right]
\Leftrightarrow I_{2s}(\mu) < \infty.$$
\end{proof}

\subsection{Laws of large numbers}
\label{subsection:law_of_large_numbers}
Given that $\mathbb{E}[J_s(P_n)] = I_s(\mu)$ above, we might expect that, in some sense.

$$J_s(P_n) \overset{n \to \infty}{\longrightarrow} I_s(\mu)$$

This is indeed the case as can be seen using the theory of $U$-statistics.

\begin{theorem}
\label{theorem:law_large_numbers}
Let $\{X_n\}$ be a collection of IID random variables where $X_n \sim \mu$ for some Borel probability measure $\mu$ on $\mathbb{R}^d$. Let $\{P_n\}$ be the random sets

$$P_n = \{X_1, ..., Xn\}$$

and define the random variables

$$U_n = J_s(P_n).$$

Then, \begin{itemize}
    \item [(1)] if $I_s(\mu) < \infty$, $U_n \to I_s(\mu)$ almost surely, and
    \item [(2)] if $I_s(\mu) = \infty$ then $U_n \to \infty$ almost surely.
\end{itemize}
\end{theorem}

\begin{proof}
Since

$$U_n = {{n}\choose{2}}^{-1} \sum_{1 \leq i < j \leq n} |x_i - x_j|^{-s},$$

$U_n$ is a $U$-statistic of the kernel $h(x,  y) = |x - y|^{-s}$.  By Theorem 5.4A of \cite{Serfling1981-je}, if $I_s(\mu) < \infty$ we get that

$$U_n \overset{n \to \infty}{\longrightarrow} J_s(P_n)$$

almost surely. If $I_s(\mu) = \infty$ then note that for $M > 0$, using $K_M$ as the kernel we get

$${{n}\choose{2}}^{-1} \sum_{1 \leq i < j \leq n} K_M(x_i, x_j) \overset{n \to \infty}{\longrightarrow} \iint K_M(x, y) d\mu(x)d\mu(y)$$

by the same theorem. With probability $1$, the above is true for all $M \in \mathbb{N}$. Since $U_n$ dominates all of the truncated statistics, we can use the monotone convergence theorem with $M \to \infty$ to get that $U_n \to \infty$ almost surely.
\end{proof}

\section{Connections to Erd\H{o}s/Falconer type problems}
\label{section:connections_to_other_problems}

In \cite{Iosevich2014}, \cite{Iosevich2019}, and other papers, the discrete energy is used to make connections between discrete and continuous problems. Given the theorems relating the discrete energy to the Riesz energy, here, we can extend these results to new situations.

\subsection{Connecting the Falconer and Erd\H{o}s Distance Problems}
\label{subsection:falconer_erdos_connections}

We briefly describe the Erd\H{o}s and Falconer distance problems here, but a more extensive discussion can be found in \cite{Iosevich2014}. Let $P \subset \mathbb{R}^d$ be a finite set and its corresponding distance set be

$$\Delta(P) = \Delta P = \{|x - y| : x, y \in P\}.$$

The Erd\H{o}s distance problem (\cite{Erdos1946-cb}) asks for a minimum size for $\Delta P$ in terms of $\# P$. It is conjectured that

$$\#(\Delta P) \gtrapprox (\#P)^{\frac{2}{d}}.$$

where $X \gtrapprox Y$ means that for all $\varepsilon > 0$, there is a $C_\varepsilon \in \mathbb{R}$ such that $C_{\varepsilon} X^{\varepsilon} \geq Y$.

The Falconer distance problem (\cite{Falconer1985-kj}) is the continuous version of the Erd\H{o}s distance problem. It starts with an arbitrary set, $P \subset \mathbb{R}^d$ and asks what the minimum value for $dim_{\mathcal{H}}(P)$ must be to guarantee that $\Delta P$ has positive (one-dimensional) Lebesgue/Hausdorff measure. The conjecture is that one needs

$$dim_{\mathcal{H}}(P) > \frac{d}{2}.$$

The best known results to date are that $\frac{5}{4}$ for $d = 2$ (\cite{Guth2020-xa}) and $\frac{d}{2} + \frac{1}{4} - \frac{1}{8d + 4}$ for $d \geq 2$ (\cite{Du2023-mf}). This remains an open and vibrant area of research.

In \cite{Iosevich2014}, these problems are connected. Below, we (slightly) generalize that connection and then use the theorems in sections \ref{section:non_randomly_drawn_points} and \ref{section:randomly_drawn_points} to extend the situation where these connections apply.

\begin{definition}
\label{definition:hausdorff_metric} The \textbf{Hausdorff distance} is a metric on closed, non-empty subsets of $\mathbb{R}^d$ defined as

$$d_H(E, F) = max\left\{ \sup_{a\in E}d(a, F), \sup_{b\in Y}d(b, X) \right\}.$$

\end{definition}

\begin{theorem}
\label{theorem:discrete_energy_facloner_erdos} Assume that the Falconer distance problem is true to the degree that there is an $s_0 > d/2$ such that if $E \subset \mathbb{R}^d$ then

$$dim_{\mathcal{H}}(E) > s_0 \Rightarrow \mathcal{L}^1( \Delta E  ) > 0.$$

Under this assumption, if $\{P_n\}$ are subsets of $[0, 1]^d$ such that $\#P_n = n$ and $\{P_n\}$ is $s$-adaptable for some $s > s_0$ then

\begin{equation}
\label{equation:erdos_distance_s_0}
\#(\Delta P_n) \gtrapprox n^{1/s_0}.
\end{equation}
\end{theorem}

\begin{proof}
The beginning of this proof is exactly the same as the proof of Theorem 3.4(a) in \cite{Iosevich2014}. Assume that equation (\ref{equation:erdos_distance_s_0}) does not hold which means that there is some $\varepsilon > 0$ and a subsequence $\{n_j\}$ such that

\begin{equation}
\label{equation:erods_contradiction}
\#(\Delta P_{n_j}) \lesssim n_j^{1/s_0 - \varepsilon}.
\end{equation}

Pick at $t \in (s_0, s)$ such that $1/s_0 - \varepsilon < 1/t$. We know that $I_t(P_{n_j})$ is uniformly bounded by some $C_t$. We take $\kappa_{n_j, 1}$ from equation \ref{equation:ball_measures} and let $K_{n_j} = supp(\kappa_{n_j}) \subset [-1, 2]^d$. We can pass to a subsequence (which we still call $K_{n_j}$) such that, for some $K_0$

$$K_{n_j} \overset{d_H}{\longrightarrow} K_0$$

and then to another subsequence (which, again, we still call $K_{n_j}$) such that for some $\kappa_0$

$$\kappa_{n_j} \overset{\star}{\longrightarrow} \kappa_0.$$

Also, as shown in \cite{Iosevich2014}, we know that

$$\Delta K_{n_j} \overset{d_H}{\longrightarrow} \Delta K_0.$$

Since $I_t(\kappa_{n_j})$ are uniformly bounded by $C_t < \infty$ we know that $I_t(\kappa_0) \leq C_t < \infty$ which means that $K_0 = supp(\kappa_0)$ has Hausdorff dimension at least $t > s_0$. Using our assumption about the Falconer distance problem, we get that

$$\mathcal{L}^1(\Delta K_0) = L > 0 \Rightarrow \mathcal{L}^1(\Delta K_{n_j}) \to L.$$

Let $\delta \in (0, 1)$ and so, for $n_j$ large enough, we know that

$$\mathcal{L}^1(\Delta K_{n_j}) > \delta L.$$

Recall that $K_{n_j}$ is just $P_{n_j}$ where each point is replaced by a ball of radius $n_j^{-1/t}$ and, so, we know that

$$\# (\Delta P_{n_j}) \geq \frac{\mathcal{L}^1(\Delta K_{n_j})}{2n_j^{-1/t}}$$

and so for big enough $n_j$

$$\# (\Delta P_{n_j}) \geq \frac{\delta L}{2n_j^{-1/t}}.$$

If equation \ref{equation:erods_contradiction} holds, however, we'd get that

$$\frac{\delta  L}{2n_j^{-1/t}} \lesssim n^{1/s_0 - \varepsilon} \Rightarrow L \lesssim n^{1/s_0 - \varepsilon - 1/t} \overset{n \to \infty}{\longrightarrow} 0$$

since we chose $t$ such that $1/s_0 - \varepsilon - 1/t < 0$. This gets us $L = 0$ which is a contradiction.
\end{proof}

With these connections we can use theorem \ref{theorem:law_large_numbers} to obtain a corollary.

\begin{corollary}
\label{corollary:prob_falconer_erdos}
Assume that the Falconer distance problem is true to the degree that there is an $s_0 > d/2$ such that if $E \subset \mathbb{R}^d$ then

\begin{equation}
\label{equation:falconer_type_assumption}
dim_{\mathcal{H}}(E) > s_0 \Rightarrow \mathcal{L}^1( \Delta E  ) > 0.
\end{equation}

Under this assumption, let $\left( [0, 1]^d, \mathcal{B}, \mu \right)$ be a Borel probability space on $\mathbb{R}^d$, $\{X_n\}$ be IID random variables drawn via $\mu$, that generate $\{P_n\}$. If $I_s(\mu) < \infty$ for some $s > s_0$ then

$$\#(\Delta P_n) \gtrapprox n^{1/s_0}.$$
\end{corollary}

In the case that $\{P_n\}$ is an $s$-adaptable sequence of sets such that $\# P_n = n$ (such as those described above) then, using the results in \cite{Guth2020-xa} and \cite{Du2023-mf} we obtain

$$\#(\Delta P_n) \gtrapprox n^{\frac{4}{5}}$$

for $d = 2$ and 

$$\#(\Delta P_n) \gtrapprox n^{\left(\frac{d}{2} + \frac{1}{4} - \frac{1}{8d + 4}\right)^{-1}} = n^{\frac{4(2d + 1)}{(2d + 1)^2 - 1}}$$

for $d \geq 3$.

\subsection{Connecting more general Erd\H{o}s/Falconer type problems}

The questions asked by the Erd\H{o}s and Falconer distance problems can be varied in interesting ways. One way this is often done is by replacing $| x - y |$ with a different $\phi:\mathbb{R}^d \times \mathbb{R}^d \to \mathbb{R}$ where for a $E \subset \mathbb{R}^d$ we let

$$\Phi(E) = \phi(E, E)$$

For example, it is shown in \cite{ESWARATHASAN20112385} that if $\phi(x, y)$ is continuous and smooth away from the diagonal and if the Monge-Ampere determinant,

$$det \begin{pmatrix}
    0 & \nabla_x \phi \\
    -(\nabla_y \phi)^T & \frac{\partial^2 \phi}{\partial x_i \partial y_j}
\end{pmatrix}$$

is non-vanishing on non-zero level sets, i.e. those of the form $\{(x, y) \in \mathbb{R}^d \times \mathbb{R}^d : \phi(x, y) = t\}$ for $t \neq 0$, then for compact $E \subset \mathbb{R}^d$ with $d > 2$, if $dim_{\mathcal{H}}(E) > \frac{d + 1}{2}$ then $\Phi(E)$ has positive Lebesgue measure.

The specific situation of $\phi(x, y) = x \cdot y$ (the standard dot product) has been studied a great deal e.g. \cite{Iosevich2021-qf} and \cite{MR4196055}. More recently, \cite{bright2024pinneddotproductset} discusses up-to-date results regarding when $\Phi(E)$ has positive Lebesgue measure as well as new results relating to the Hausdorff dimension of $\Phi(E)$ and when $\Phi(E)$ has non-empty interior.

The argument for theorem \ref{theorem:discrete_energy_facloner_erdos} above can be generalized.

\begin{theorem}
\label{theorem:discrete_energy_falconer_erdos_general}
Let $\phi:\mathbb{R}^d \times \mathbb{R}^d \to \mathbb{R}$ be locally Lipschitz with Lipschitz constant $\lambda_d$ (i.e. the Lipschitz constant can depend on $d$). Assume a Falconer-type bound is true i.e. there is some $s_0$ such that if $E \subset \mathbb{R}^d$ then

$$dim_{\mathcal{H}}(E) > s_0 \Rightarrow \mathcal{L}^1(\Phi(E)) > 0.$$

Under this assumption, if $\{P_n\}$ are subsets of $[0, 1]^d$ such that $\#P_n = n$ and $\{P_n\}$ is $s$-adaptable for some $s > s_0$ then

\begin{equation}
\label{equation:erdos_general_s_0}
\#(\Phi(P_n)) \gtrapprox n^{1/s_0}.
\end{equation}
\end{theorem}

\begin{proof}
The proof is largely the same as the proof of theorem \ref{theorem:discrete_energy_facloner_erdos}. Assume that equation \ref{equation:erdos_general_s_0} does not hold which means that there is some $\varepsilon > 0$ and a $\{n_j\}_{j = 1}^{\infty}$ such that

\begin{equation}
\label{equation:erods_contradiction_s_0}
\#(\Phi( P_{n_j})) \lesssim n^{1/s_0 - \varepsilon}.
\end{equation}

Pick at $t \in (s_0, s)$ such that $1/s_0 - \varepsilon < 1/t$. We know that $I_t(P_{n_j})$ is uniformly bounded by some $C_t$. We take $\kappa_{n_j, 1}$ from equation \ref{equation:ball_measures} and let $K_{n_j} = supp(\kappa_{n_j}) \subset [-1, 2]^d$. We can pass to a subsequence (which we still call $K_{n_j}$) such that, for some $\widetilde{K_0}$

$$K_{n_j} \overset{d_H}{\longrightarrow} \widetilde{K_0}$$

and then to another subsequence (which, again, we still call $K_{n_j}$) such that for some $\kappa_0$

$$\kappa_{n_j} \overset{\star}{\longrightarrow} \kappa_0.$$

We can then see that $\phi(K_{n_j}) \to \phi(K_0)$. To see this, note that if $F_n \overset{H}{\longrightarrow} F$ then for any $\alpha > 0$, and letting $(S)_{\alpha}$ denote the $\alpha$-neighborhood of $S \subset \mathbb{R}^d$, if $n$ is large enough, then $(F)_{\alpha} \supseteq F_n$ and $(F_n)_{\alpha} \supseteq F$. Thus, $\phi(F_{\alpha}) \supseteq \phi(F_n)$ and $\phi((F_n)_{\alpha}) \supseteq \phi(F)$. Now, we note that for some set $G$ and sufficiently small $\alpha$,

$$\Phi(G_{\alpha})
= \bigcup_{x, y \in G} \phi(B_{\alpha}(x), B_{\alpha}(y))
\subset \bigcup_{x, y \in G} \phi(x, y)_{2 \alpha \lambda_d}
= \Phi(G)_{2 \alpha \lambda_d}.$$

Since $I_t(\kappa_{n_j})$ are uniformly bounded by $C_t < \infty$ we know that $I_t(\kappa_0) \leq C_t < \infty$ which means that $K_0 = supp(\mu_0)$ has Hausdorff dimension at least $t > s_0$. Using our assumption in equation \ref{equation:falconer_type_assumption}

$$\mathcal{L}^1(\Phi(K_0)) = L > 0 \Rightarrow \mathcal{L}^1(\Phi(K_{n_j})) \to L.$$

Let $\delta \in (0, 1)$ and so, for $n_j$ large enough, we know that

$$\mathcal{L}^1(\Phi K_{n_j}) > \delta L.$$

Recall that $K_{n_j}$ is just $P_{n_j}$ where each point is replaced by a ball of radius $n_j^{-1/t}$ and, so, we know that, for big enough $n_j$

$$\# (\Phi P_{n_j}) \geq \frac{\delta L}{2 \lambda_d n_j^{-1/t}}.$$

If equation \ref{equation:erods_contradiction_s_0} holds, however, we'd get that

$$\frac{\delta  L}{2 \lambda n_j^{-1/t}} \lesssim n^{1/s_0 - \varepsilon} \Rightarrow L \lesssim n^{1/s_0 - \varepsilon - 1/t} \to 0$$

since we chose $t$ such that $1/s_0 - \varepsilon - 1/t < 0$. This gets us $L = 0$ which is a contradiction.

\end{proof}

We can apply theorem \ref{theorem:discrete_energy_falconer_erdos_general} to $\phi(x, y) = x \cdot y$.

\begin{corollary}
\label{corollary:discrete_energy_falconer_erdos_dot_product}
For a set $E \subset \mathbb{R}^d$, let $\Pi(E) = \{x \cdot y : x, y \in E\}$.  Assume that the ``Falconer dot product problem'' is true to the degree that there is an $s_0 > d/2$ such that if $E \subset \mathbb{R}^d$ then

$$dim_{\mathcal{H}}(E) > s_0 \Rightarrow \mathcal{L}^1( \Pi(E)  ) > 0.$$

Under this assumption, if $\{P_n\}$ are subsets of $[0, 1]^d$ such that $\#P_n = n$ and $\{P_n\}$ is $s$-adaptable for some $s > s_0$ then

\begin{equation}
\label{equation:erdos_general_s_0_result}
\#(\Pi(P_n)) \gtrapprox n^{1/s_0}.
\end{equation}
\end{corollary}

\begin{proof}
For $x, y \in [0, 1]^d$, $\alpha \in [0, 1]$, and $u, v \in B_{\alpha}(0)$

$$|(x + u) \cdot (y + v) - x \cdot y| = |x \cdot v + y \cdot u + u \cdot v| \leq \alpha^2 + 2\alpha \leq 3 \alpha.$$

So, applying theorem \ref{theorem:discrete_energy_falconer_erdos_general}, we get the desired result.
\end{proof}

In addition, we can get a corollary similar to corollary \ref{corollary:prob_falconer_erdos} for the dot product.

\printbibliography

\end{document}